\documentclass[12pt]{article}
\usepackage{amsmath}
\usepackage{color,amsfonts,amssymb, amsthm, bm}
\usepackage{amsfonts,epsf}
\usepackage{latexsym}
\usepackage{graphicx}
\usepackage{array,float}
\usepackage{tikz}
\usetikzlibrary{calc, positioning, fit, shapes.geometric}
\pgfdeclarelayer{background}
\pgfsetlayers{background,main}
\usepackage{enumerate}
\usepackage{diagbox}
\usepackage{hyperref}
\usepackage{url}

\hypersetup{
  colorlinks   = true, 
  urlcolor     = blue, 
  linkcolor    = blue, 
  citecolor   = blue 
}

\newtheorem{theorem}{Theorem}[section]

\newtheorem{proposition}[theorem]{Proposition}
\newtheorem{corollary}[theorem]{Corollary}

\newtheorem{problem}[theorem]{Problem}

\newcommand{\cp}{\,\square\,}

\DeclareMathOperator{\rad}{rad}
\DeclareMathOperator{\diam}{diam}
\DeclareMathOperator{\ecc}{ecc}

\begin{document}
\title{Generalized Pell graphs}

\author{
Vesna Ir\v{s}i\v{c}$^{1, 2}$
\and
Sandi Klav\v zar$^{1, 2, 3}$
\and
Elif Tan$^{4}$
}

\date{}

\maketitle
\begin{center}
$^1$ Faculty of Mathematics and Physics,  University of Ljubljana, Slovenia\\
{\tt vesna.irsic@fmf.uni-lj.si, sandi.klavzar@fmf.uni-lj.si}\\
\medskip

$^2$ Institute of Mathematics, Physics and Mechanics, Ljubljana, Slovenia\\
\medskip

$^3$ Faculty of Natural Sciences and Mathematics,  University of Maribor, Slovenia\\
\medskip

$^4$ Department of Mathematics, Ankara University, Ankara, Turkey  \\
{\tt etan@ankara.edu.tr}\\

\end{center}

\begin{abstract}
In this paper, generalized Pell graphs $\Pi _{n,k}$, $k\ge 2$, are introduced. The special case of $k=2$ are the Pell graphs $\Pi _{n}$ defined earlier by Munarini. Several metric, enumerative, and structural properties of these graphs are established. The generating function of the number of edges of  $\Pi _{n,k}$ and the generating function of its cube polynomial are determined. The center of $\Pi _{n,k}$ is explicitly described; if $k$ is even, then it induces the Fibonacci cube $\Gamma_{n}$. It is also shown that $\Pi _{n,k}$ is a median graph, and that $\Pi _{n,k}$ embeds into a Fibonacci cube. 
\end{abstract}



\section{Introduction}

The Fibonacci sequence is one of the most famous sequences in mathematics.
The $n$th Fibonacci number $F_{n}$ is defined by $F_{n}=F_{n-1}+F_{n-2}$, $n\geq 2$, with initial values $F_{0}=0$ and $F_{1}=1.$ Fibonacci numbers and their generalizations
have many interesting properties and different applications in science and art.
There are several generalizations of Fibonacci sequence. One among them is the $k$%
-Fibonacci sequence defined by Falcon and Plaza \cite{falcon-2007} for a positive integer $k$ as 
\begin{equation}
F_{n,k}=kF_{n-1,k}+F_{n-2,k},\ n\geq 2,  \label{eq:F_n,k}
\end{equation}%
with initial values $F_{0,k}=0$ and $F_{1,k}=1$. The first few terms of the $k$-Fibonacci sequence are $0$, $1$, $k$, $k^{2}+1$, $k^{3}+2k$, $k^{4}+3k^{2}+1$, $k^{5}+4k^{3}+3k$. If $k=1$, then~\eqref{eq:F_n,k} reduces to the Fibonacci sequence $\left \{
F_{n}\right \}_{n\geq 0} $, and if $k=2$, then~\eqref{eq:F_n,k}  reduces to the Pell sequence $\left \{P_{n}\right \}_{n \geq 0} $, where $P_0 = 0$, $P_1 = 1$, and $P_n = 2P_{n-1} + P_{n-2}$ for $n\ge 2$. The generating function of the $k$-Fibonacci sequence is given by%
\begin{equation}
\label{eq:gen-fn-k-fib}
F(t)=\sum_{n\geq 0}F_{n,k}t^{n}=\frac{t}{1-kt-t^{2}}\,.
\end{equation}
For more on these sequences, we refer the reader to \cite{falcon-2011}. It should also be noted that the $k$-Fibonacci numbers defined and used here are not to be confused with the $k$-generalized Fibonacci numbers (or generalized order-$k$ Fibonacci numbers) that are defined as
\begin{equation*}
F_{n,k}=F_{n-1,k}+F_{n-2,k}+···+ F_{n-k,k},\ n\geq k,
\end{equation*}%
with the appropriate initial conditions, see \cite{kalman-1982}.

Fibonacci cubes which were introduced in 1993 in~\cite{hsu-1993}. They are closely related to the Fibonacci sequence. They have found numerous applications elsewhere and are also extremely interesting in their own right. The state of the art on Fibonacci cubes and related classes of graphs has been collected in the book~\cite{FibonacciBook-2023} published in 2023, the research in this direction is still ongoing, see~\cite{egecioglu-2023+, kirgizov-2022, klavzar-2023}. On the other hand, motivated by the Pell sequence, in 2019 Munarini introduced Pell graphs~\cite{munarini-2019}. Pell graphs have been further investigated in~\cite{ozer-2023}. In this paper, based on the definition~\eqref{eq:F_n,k} we introduce generalized Pell graphs such that for each $k\ge 2$ their construction reflects the recursion ~\eqref{eq:F_n,k}. 

The rest of the paper is organized as follows. In the next section, we define the concepts discussed in this paper and introduce the required notation. In Section~\ref{sec:basic-properties} the two-parameter generalized Pell graphs $\Pi _{n,k}$ are formally defined and their fundamental structure described. Among other results we determine the generating function of the number of edges of  $\Pi _{n,k}$ and observe that they are traceable, that is, they contain Hamiltonian paths. In Section~\ref{sec:rad-diam} we determine the radius and the diameter of $\Pi _{n,k}$. Furthermore, we describe the structure of the center of $\Pi _{n,k}$. Interestingly, if $k$ is even, then the center of $\Pi _{n,k}$ induces the Fibonacci cube $\Gamma_{n}$. In Section~\ref{sec:additional} additional properties of $\Pi _{n,k}$ are established: the generating function of its cube polynomial, distribution of its degrees, the fact that $\Pi _{n,k}$ is a median graph, and that $\Pi _{n,k}$ embeds into a Fibonacci cube. We conclude the paper with some remarks on a similar project undertaken independently by Do\v{s}li\'{c} and Podrug~\cite{podrug-2023} and with some open problems.

\section{Preliminaries}

Let $G=(V(G),E(G))$ be a graph where $V(G)$ is a set of vertices and $E(G)$ is a set of edges consisting of unordered pairs of vertices. The numbers of vertices and edges in $G$ are called the \textit{order} and the \textit{size} of $G$, respectively. The degree $\deg(u)$ of a vertex $u\in V(G)$ is the number of edges incident with it in $G$. As usual, we use $\Delta (G)$ and $\delta (G)$ to denote the maximum and the minimum degree of $G$, respectively. The subgraph induced by $X \subseteq V(G)$ is denoted by $G[X]$.

The \emph{distance} $d(u,v)$ between vertices $u$ and $v$ of a graph $G$ is the number of edges on a shortest $u,v$-path. The \emph{eccentricity} $\ecc(u)$ of a vertex $u \in V(G)$ is the maximum distance between $u$ and all other vertices of $G$. The \emph{radius} $\rad(G)$ and the \emph{diameter} $\diam(G)$ of $G$ are the minimum and the maximum eccentricity of the vertices of $G$, respectively. The \emph{center} $C(G)$ of $G$ is the set of vertices $u \in V(G)$ with $\ecc(u) = \rad(G)$. The \emph{periphery} of $G$ is defined as the set of vertices $u \in V(G)$ with $\ecc(u) = \diam(G)$.

The \emph{Cartesian product} $G\cp H$ of graphs $G$ and $H$ has vertices $V(G)\times
V(H) $ and edges $(g,h)(g^{\prime },h^{\prime })$, where either $g=g^{\prime
}$ and $hh^{\prime }\in E(H)$, or $h=h^{\prime }$ and $gg^{\prime }\in E(G)$. The \emph{$r$-cube} $Q_{r}$, $r\geq 1$, is the graph with $V(Q_{r})=\{0,1\}^{r}$, with an edge between two vertices if and only
if they differ in exactly one coordinate. That is, if $x=(x_{1},\ldots
,x_{r})$ and $y=(y_{1},\ldots ,y_{r})$ are vertices of $Q_{r}$, then $xy\in
E(Q_{r})$ if and only if there exists $j\in \lbrack r]=\{1,\ldots ,r\}$ such
that $x_{j}\neq y_{j}$ and $x_{i}=y_{i}$ for every $i\neq j$. The $r$-cube $Q_{r}$, $r\geq 2$, can also be described as the Cartesian product $Q_{r-1}\cp K_{2}$.

Let $\mathcal{F}_{n}$ denote the set of Fibonacci strings of length $n$, that is, binary strings of length $n$ that contain no consecutive $1$s. Then $\mathcal{F}_{0} = \{ \varepsilon \}$, $\mathcal{F}_{1} = \{0,1\}$, and if $n\ge 2$, then 
$$\mathcal{F}_{n+2} = 0\mathcal{F}_{n+1} + 10\mathcal{F}_{n}\,,$$
where $+$ denotes the disjoint union of sets. Consequently, $|\mathcal{F}_{n}| = F_{n+2}$. The \emph{Fibonacci cube} $\Gamma _{n}$, $%
n\geq 1$, is the graph with $V(\Gamma_n) = \mathcal{F}_{n}$ in which two vertices are adjacent if they differ in a single coordinate. Hence $|V(\Gamma _{n})| = F_{n+2}$. Note that the strings $0\mathcal{F}_{n+1}$ in $\mathcal{F}_{n+2}$ induce a subgraph isomorphic to $\Gamma _{n+1}$ and the strings $10\mathcal{F}_{n}$ induce a subgragraph isomorphic to $\Gamma _{n}$. 

If $w$ is a word over an alphabet $\Sigma$ and $a\in \Sigma$, then a {\em run} of $a$s is a maximal subword of $w$ such that all of its letters are $a$ (sometimes called a \emph{block}). (For the research on runs in binary strings and the so-called Fibonacci-run graphs see~\cite{baril-2022, egecioglu-irsic-2021a, egecioglu-irsic-2021b, wei-2023}.) A {\em Pell string} is a word over the alphabet $T=\{0,1,2\}$ such that there are no runs of $2s$ of odd length~\cite{munarini-2019}. Equivalently, a Pell string is a word over the alphabet $T^{\prime }=\{0,1,22\}$. Let $\mathcal{P}_{n}$ denote the set of Pell strings of length $n$. Then $\mathcal{P}_{0} = \{ \varepsilon \}$, $\mathcal{P}_{1} = \{0,1\}$ and for $n\ge 0$, 
$$\mathcal{P}_{n+2} = 0\mathcal{P}_{n+1}+1\mathcal{P}_{n+1}+22\mathcal{P}_{n}\,.$$
Thus $|\mathcal{P}_{n}|=P_{n+1}$. The {\em Pell graph} $\Pi _{n}$, $n\ge 0$, has $V(\Pi_n) = \mathcal{P}_{n}$ and two vertices in $\Pi_n$ are adjacent whenever one of them can be obtained from the other by replacing a $0$ with a $1$ (or vice versa), or by replacing a factor $11$ with $22$ (or vice versa). Then $\Pi_{0}=K_{1}$, $\Pi _{1}=K_{2}$, and $|V(\Pi_{n})| = P_{n+1}$. Furthermore, the number of edges in $\Pi_{n}$ satisfies $|E(\Pi _{n})|=|E(\Pi _{n-1})|+|E(\Pi _{n-1})|+|E(\Pi _{n-2})|+P_{n+1}+P_{n}$. In Figure~\ref{fig:pell} the first four Pell graphs are drawn. See~\cite{munarini-2019} for more on Pell graphs.

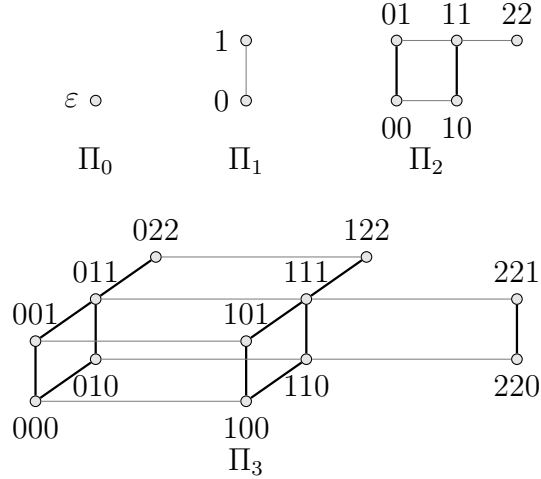
\begin{figure}[ht!]
    \centering
    \begin{tikzpicture}
    [scale=0.8,
    vert/.style={circle, draw, fill=black!10, inner sep=0pt, minimum width=4pt}, 
    central/.style={circle, draw, fill=black, inner sep=0pt, minimum width=4pt},
    ]

    \begin{scope}[xshift=1cm]
    \node[vert, label=left:$\varepsilon$] (0) at (0,0) {};

    \node (p1) at (0,-1) {$\Pi_0$};
    \end{scope}

    \begin{scope}[xshift=3.5cm]
    \node[vert, label=left:$0$] (0) at (0,0) {};
    \node[vert, label=left:$1$] (1) at (0,1) {};

    \path[gray] (0) edge (1);

    \node (p1) at (0,-1) {$\Pi_1$};
    \end{scope}

    \begin{scope}[xshift=6cm]
    \node[vert, label=below:$00$] (00) at (0,0) {};
    \node[vert, label=above:$01$] (01) at (0,1) {};
    \node[vert, label=below:$10$] (10) at (1,0) {};
    \node[vert, label=above:$11$] (11) at (1,1) {};
    \node[vert, label=above:$22$] (22) at (2,1) {};

    \path[line width=0.3mm] (00) edge (01);
    \path[line width=0.3mm] (10) edge (11);
    \path[gray] (00) edge (10);
    \path[gray] (01) edge (11);
    \path[gray] (11) edge (22);

    \node (p2) at (0.5,-1) {$\Pi_2$};
    \end{scope}
    
    \begin{scope}[yshift=-5cm]
    \node[vert, label=below:$000$] (000) at (0,0) {};
    \node[vert, label=above:$001$] (001) at (0,1) {};
    \node[vert, label=below:$010$] (010) at (1,0.7) {};
    \node[vert, label=above:$011$] (011) at (1,1+0.7) {};
    \node[vert, label=above:$022$] (022) at (2,1+2*0.7) {};

    \path[line width=0.3mm] (000) edge (001);
    \path[line width=0.3mm] (010) edge (011);
    \path[line width=0.3mm] (000) edge (010);
    \path[line width=0.3mm] (001) edge (011);
    \path[line width=0.3mm] (011) edge (022);

    \node[vert, label=below:$100$] (100) at (3.5,0) {};
    \node[vert, label=above:$101$] (101) at (3.5,1) {};
    \node[vert, label=below:$110$] (110) at (4.5,0.7) {};
    \node[vert, label=above:$111$] (111) at (4.5,1+0.7) {};
    \node[vert, label=above:$122$] (122) at (5.5,1+2*0.7) {};

    \path[line width=0.3mm] (100) edge (101);
    \path[line width=0.3mm] (110) edge (111);
    \path[line width=0.3mm] (100) edge (110);
    \path[line width=0.3mm] (101) edge (111);
    \path[line width=0.3mm] (111) edge (122);

    \node[vert, label=below:$220$] (220) at (8,0.7) {};
    \node[vert, label=above:$221$] (221) at (8,1+0.7) {};

    \path[line width=0.3mm] (220) edge (221);

    \node (p2) at (3.5,-1) {$\Pi_3$};

    \foreach \x in {0,1}
        \foreach \y in {0,1}{
        \path[gray] (0\x\y) edge (1\x\y);       
        }
        
    \path[gray] (022) edge (122);
    \path[gray] (110) edge (220);
    \path[gray] (111) edge (221);

    \end{scope}
    \end{tikzpicture}
    \caption{Pell graphs $\Pi_n$ for $n \in \{0,1,2,3\}$.}
    \label{fig:pell}
\end{figure}

\section{Generalized Pell graphs and their basic properties}
\label{sec:basic-properties}

Motivated by the facts from the introduction, we now define generalized Pell graphs as follows. 

If $k\geq 2$, then a \emph{generalized Pell string} is a string over the alphabet $\{0,1,\ldots ,k-1,kk\}$. Note that a generalized Pell string with $k=2$ is a Pell string, and that if $\alpha$ is a generalized Pell string, then each run of $k$s is of even length. If $n \ge 0$ and $k\ge 2$, then let $\mathcal{F}_{n,k}$ be the set of the generalized Pell strings of length $n$. Clearly, $\mathcal{F}_{0,k}=\{ \varepsilon \}$ and $\mathcal{F}_{1,k}=\{0,1,\ldots ,k-1\}$, while for $n\ge 2$ we have 
\begin{equation*}
\mathcal{F}_{n,k}=0\mathcal{F}_{n-1,k} + 1\mathcal{F}_{n-1,k} + \cdots +\left( k-1\right) \mathcal{F}_{n-1,k}+kk\mathcal{F}_{n-2,k}\,.
\end{equation*}%
Therefore, $|\mathcal{F}_{n,k}|=F_{n+1,k},$ where the values $F_{n,k}$ are defined in~\eqref{eq:F_n,k}. 

Now, if $n\ge 0$ and $k\ge 2$, then the {\em generalized Pell graph} $\Pi _{n,k}$ has the vertex set $V(\Pi _{n,k}) = \mathcal{F}_{n,k}$ and two vertices being adjacent whenever one of them can be obtained from the other by either replacing an $i$ with an $i+1$ (or vice versa), where $i \in \{0,1,\ldots, k-2\}$, or by replacing one factor $\left( k-1\right)
\left( k-1\right) $ with $kk$ (or vice versa) in such a way that the new string is again a generalized Pell string. Note that $\Pi_{n,2}$ = $\Pi _{n}$. In Figure~\ref{fig:pi-k-3} the generalized Pell graphs $\Pi_{n,3}$, $n \in \{0,1,2,3\}$, are shown. 

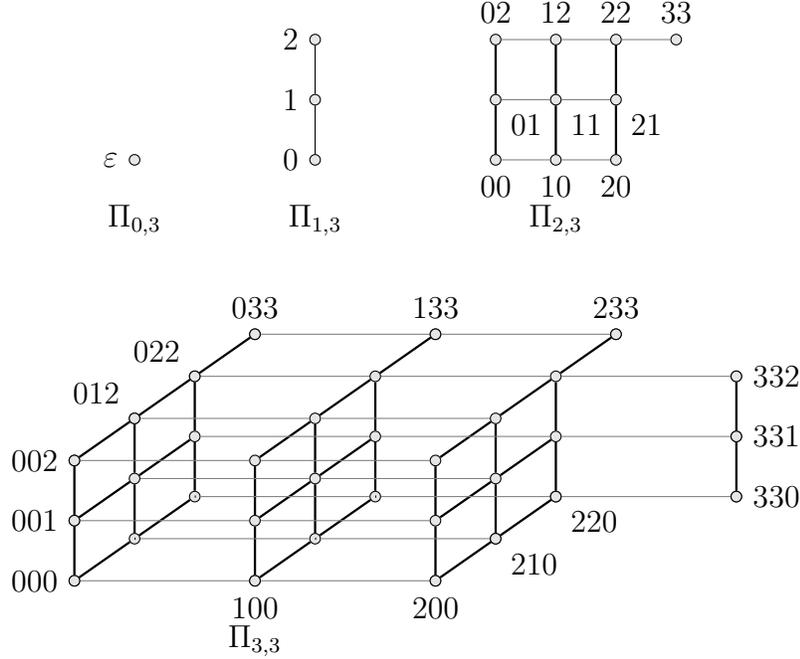
\begin{figure}[hbt!]
    \centering
    \begin{tikzpicture}
    [scale=0.8,
    vert/.style={circle, draw, fill=black!10, inner sep=0pt, minimum width=4pt}, 
    central/.style={circle, draw, fill=black, inner sep=0pt, minimum width=4pt},
    ]

    \begin{scope}[xshift=1cm]
    \node[vert, label=left:$\varepsilon$] (0) at (0,0) {};

    \node (p03) at (0,-1) {$\Pi_{0,3}$};
    \end{scope}

    \begin{scope}[xshift=4cm]
    \node[vert, label=left:$0$] (0) at (0,0) {};
    \node[vert, label=left:$1$] (1) at (0,1) {};
    \node[vert, label=left:$2$] (2) at (0,2) {};

    \path[gray] (0) edge (1);
    \path[gray] (1) edge (2);

    \node (p13) at (0,-1) {$\Pi_{1,3}$};
    \end{scope}

    \begin{scope}[xshift=7cm]
    \node[vert, label=below:$00$] (00) at (0,0) {};
    \node[vert, label=below right:$01$] (01) at (0,1) {};
    \node[vert, label=above:$02$] (02) at (0,2) {};
    \node[vert, label=below:$10$] (10) at (1,0) {};
    \node[vert, label=below right:$11$] (11) at (1,1) {};
    \node[vert, label=above:$12$] (12) at (1,2) {};
    \node[vert, label=below:$20$] (20) at (2,0) {};
    \node[vert, label=below right:$21$] (21) at (2,1) {};
    \node[vert, label=above:$22$] (22) at (2,2) {};
    \node[vert, label=above:$33$] (33) at (3,2) {};

    \path[line width=0.3mm] (00) edge (01);
    \path[line width=0.3mm] (01) edge (02);
    \path[line width=0.3mm] (10) edge (11);
    \path[line width=0.3mm] (11) edge (12);
    \path[line width=0.3mm] (20) edge (21);
    \path[line width=0.3mm] (21) edge (22);
    \path[gray] (00) edge (10);
    \path[gray] (01) edge (11);
    \path[gray] (02) edge (12);
    \path[gray] (10) edge (20);
    \path[gray] (11) edge (21);
    \path[gray] (12) edge (22);
    \path[gray] (22) edge (33);

    \path (0) edge (1);
    \path (1) edge (2);

    \node (p23) at (1,-1) {$\Pi_{2,3}$};
    \end{scope}
    
    \begin{scope}[yshift=-7cm]
    \foreach \p in {0}{
        \foreach \z in {0,1,2}
            \foreach \y in {0,1,2}
                \foreach \x in {0,1,2}
                    \node[vert] (\p+\z+\y+\x) at (3*\z+\x,\y+0.7*\x-6*\p) {};
        \foreach \z in {0,1,2}
            \node[vert] (\p+\z+3+3) at (3*\z+3,2+0.7*3-6*\p) {};
        \foreach \x in {0,1,2}
            \node[vert] (\p+3+3+\x) at (3*3+2, \x+0.7*2-6*\p) {};
            
        \foreach \z in {0,1,2}
            \foreach \y in {0,1,2}
                \foreach \x [remember=\x as \lastx (initially 0)] in {1,2}
                    \path[line width=0.3mm] (\p+\z+\y+\x) edge (\p+\z+\y+\lastx);

        \foreach \z in {0,1,2}
            \foreach \y [remember=\y as \lasty (initially 0)] in {0,1,2}
                \foreach \x in {0,1,2}
                    \path[line width=0.3mm] (\p+\z+\y+\x) edge (\p+\z+\lasty+\x);
                    
        \foreach \z [remember=\z as \lastz (initially 0)] in {1,2}
            \foreach \y in {0,1,2}
                \foreach \x in {0,1,2}
                    \path[gray] (\p+\z+\y+\x) edge (\p+\lastz+\y+\x);

        \foreach \z [remember=\z as \lastz (initially 0)] in {1,2}
            \path[gray] (\p+\z+3+3) edge (\p+\lastz+3+3);

        \foreach \z in {0,1,2}
            \path[line width=0.3mm] (\p+\z+2+2) edge (\p+\z+3+3);

        \foreach \x [remember=\x as \lastx (initially 0)] in {1,2}
            \path[line width=0.3mm] (\p+3+3+\x) edge (\p+3+3+\lastx);

        \foreach \x in {0,1,2}
            \path[gray] (\p+2+\x+2) edge (\p+3+3+\x);
        }

        \node[vert, label=left:$000$] (000) at (0,0) {};
        \node[vert, label=left:$001$] (001) at (0,1) {};
        \node[vert, label=left:$002$] (002) at (0,2) {};
        \node[vert, label=above:$033$] (033) at (3,2+0.7*3) {};
        \node[vert, label=above:$133$] (133) at (6,2+0.7*3) {};
        \node[vert, label=above:$233$] (233) at (9,2+0.7*3) {};
        \node[vert, label=right:$330$] (330) at (11,1.4) {};
        \node[vert, label=right:$331$] (331) at (11,1+1.4) {};
        \node[vert, label=right:$332$] (332) at (11,2+1.4) {};
        \node[vert, label=below:$100$] (100) at (3,0) {};
        \node[vert, label=below:$200$] (200) at (6,0) {};
        \node[vert, label=below right:$210$] (210) at (7,0.7) {};
        \node[vert, label=below right:$220$] (220) at (8,1.4) {};
        \node[vert, label=above left:$012$] (012) at (1,2.7) {};
        \node[vert, label=above left:$022$] (022) at (2,3.4) {};
        
        \node (p33) at (3,-1) {$\Pi_{3,3}$};
        \end{scope}
    \end{tikzpicture}
    \caption{Generalized Pell graphs $\Pi_{n,3}$ for $n \in \{0,1,2,3\}$. To make the figure transparent, not all vertices are labeled.}
    \label{fig:pi-k-3}
\end{figure}

The way we defined generalized Pell graphs appeared to us as the (most) natural generalization of Pell graphs such that the order of the graph is counted by the $k$-Fibonacci sequence. But there are other ways to extend Pell graphs, for instance to use the alphabet $\{0, 1, 22, 33, \ldots, kk\}$. In this case, however, the number of vertices $v_n$ would satisfy $v_n = 2v_{n-1}+(k-1)v_{n-2}$, and not the recursion from~\eqref{eq:F_n,k}.

The \emph{fundamental decomposition} of the generalized Pell graph $\Pi _{n,k}$ is the following. Note that each of the induced subgraphs $\Pi_{n,k}[j\mathcal{F}_{n-1,k}]$, $j \in \{0,\ldots,k-1\}$, is isomorphic to $\Pi_{n-1,k}$ and it is denoted by $j \Pi_{n-1,k}$. In addition, the induced subgraph $\Pi_{n,k}[kk \mathcal{F}_{n-2,k}]$ is isomorphic to $\Pi_{n-2,k}$ and denoted by $kk \Pi_{n-2,k}$. Then it is straightforward to see that $\Pi_{n,k}[\bigcup_{j=0}^{k-1} j\Pi_{n-1,k}]$ is isomorphic to the Cartesian product of $\Pi_{n-1,k}$ and the path on $k$ vertices. Additionally, each vertex from $kk \Pi_{n-2,k}$ has exactly one neighbor in $(k-1) \Pi_{n-1,k}$. For $n \geq 2$ we formally denote this fundamental decomposition as follows:
\begin{equation}
\Pi _{n,k}=0\Pi _{n-1,k} \oplus 1\Pi _{n-1,k} \oplus \cdots \oplus \left( k-1\right) \Pi
_{n-1,k} \odot kk\Pi _{n-2,k},  \label{decom}
\end{equation}%
with $\Pi _{0,k}=K_{1}$ and $\Pi _{1,k}$ is the path on $k$ vertices. See Figure~\ref{fig:decomposition}. 

\begin{figure}[htb!]
    \centering
    \begin{tikzpicture}[scale=0.7]
        \draw[line width = 0.4mm] (0,0) ellipse (5cm and 0.8cm);
        \node (0) at (0,0) {$0 \Pi_{n-1,k}$};

        \draw[line width = 0.4mm] (0,-3) ellipse (5cm and 0.8cm);
        \node (1) at (0,-3) {$1 \Pi_{n-1,k}$};

        \draw[line width = 0.4mm] (0,-7) ellipse (5cm and 0.8cm);
        \node (k-1) at (-2.5,-7) {$(k-1) \Pi_{n-1,k}$};

        \draw[line width = 0.3mm] (2,-7) ellipse (2.5cm and 0.5cm);
        \node (2) at (2,-7) {\scriptsize{$(k-1) (k-1) \Pi_{n-2,k}$}};

        \draw[line width = 0.3mm] (2,-10) ellipse (2.5cm and 0.5cm);
        \node (kk) at (2,-10) {\scriptsize{$kk \Pi_{n-2,k}$}};

        \draw (-4,-0.2) -- (-4,-2.8);
        \draw (-3,-0.2) -- (-3,-2.8);
        \node (d1) at (0,-1.5) {$\bullet\quad\bullet\quad\bullet$};
        \draw (3,-0.2) -- (3,-2.8);
        \draw (4,-0.2) -- (4,-2.8);

        \draw (-4,-3.2) -- (-4,-4.5);
        \draw (-3,-3.2) -- (-3,-4.5);
        \node (d1) at (0,-4.8) {$\bullet\quad\bullet\quad\bullet$};
        \draw (3,-3.2) -- (3,-4.5);
        \draw (4,-3.2) -- (4,-4.5);

        \node (dv1) at (-4, -4.8) {$\vdots$};
        \node (dv1) at (-3, -4.8) {$\vdots$};
        \node (dv1) at (3, -4.8) {$\vdots$};
        \node (dv1) at (4, -4.8) {$\vdots$};

        \draw (-4,-5.3) -- (-4,-6.8);
        \draw (-3,-5.3) -- (-3,-6.7);
        \draw (3,-5.3) -- (3,-6.7);
        \draw (4,-5.3) -- (4,-6.8);

        \draw (-0.2,-7.1) -- (-0.2,-9.9);
        \draw (0.8,-7.2) -- (0.8,-9.9);
        \node (d2) at (1.8, -8.5) {$\ldots$};
        \draw (3,-7.2) -- (3,-9.9);
        \draw (4,-7.1) -- (4,-9.9);
    \end{tikzpicture}
    \caption{The fundamental decomposition of $\Pi_{n,k}$.}
    \label{fig:decomposition}
\end{figure}
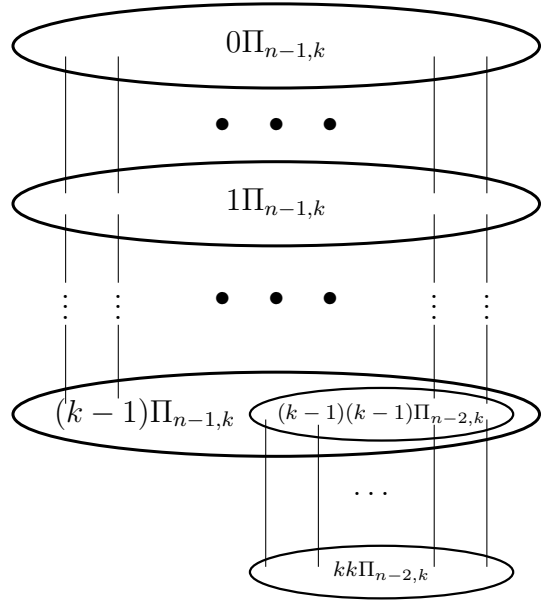

Since the generalized Pell graph $\Pi _{n,k}$ is defined on the vertex set $%
\mathcal{F}_{n,k},$ the number of vertices of $\Pi _{n,k}$ is $F_{n+1,k}$.

From the fundamental decomposition of $\Pi _{n,k}$ in (\ref{decom}), the edges of $\Pi _{n,k}$ are of the following four types:
\begin{enumerate}[(i)]
    \item edges from $k$ copies of $\Pi _{n-1,k},$
    \item edges from $\Pi _{n-2,k},$
    \item the link edges between the vertices in
the $k$ copies of $\Pi _{n-1,k}$, and
    \item the link edges between the vertices in
 $kk\Pi _{n-2,k}$ and $(k-1)(k-1) \Pi _{n-2,k}.$
\end{enumerate}

Thus the number of edges can be obtained by the following recurrence
relation, for\ $n\geq 2$%
\begin{equation*}
|E(\Pi _{n,k})|=k|E(\Pi _{n-1,k})|+|E(\Pi _{n-2,k})|+\left( k-1\right)
F_{n,k}+F_{n-1,k}
\end{equation*}%
or (by using the recurrence relation of $k$-Fibonacci numbers)%
\begin{equation}
|E(\Pi _{n,k})|=k|E(\Pi _{n-1,k})|+|E(\Pi _{n-2,k})|+F_{n+1,k}-F_{n,k}
\label{edges}
\end{equation}%
with $|E(\Pi _{0,k})|=0$ and $|E(\Pi _{1,k})|=k-1.$

\begin{proposition}
\label{prop:gf-edges}
The generating function of the number of edges in $\Pi _{n,k}$ is%
\begin{equation*}
\sum_{n\geq 0}|E(\Pi _{n,k})|t^{n}=\frac{\left( k-1+t\right) t}{\left(
1-kt-t^{2}\right) ^{2}}.
\end{equation*}
\end{proposition}

\begin{proof}
Denote the generating function of the sequence of the number of
edges in $\Pi _{n,k}$ by $E(t)$. From (\ref{edges}), we have%
\begin{eqnarray*}
E(t) &=&\sum_{n\geq 0}|E(\Pi _{n,k})|t^{n} \\
&=&|E(\Pi _{0,k})|+|E(\Pi _{1,k})|t+\sum_{n\geq 2}|E(\Pi _{n,k})|t^{n} \\
&=&\left( k-1\right) t+\sum_{n\geq 2}\left( k|E(\Pi _{n-1,k})|+|E(\Pi
_{n-2,k})|+F_{n+1,k}-F_{n,k}\right) t^{n} \\
&=&\left( k-1\right) t+k\sum_{n\geq 2}|E(\Pi _{n-1,k})|t^{n}+\sum_{n\geq
2}|E(\Pi _{n-2,k})|t^{n}+ \\
& & \sum_{n\geq 2}F_{n+1,k}t^{n}-\sum_{n\geq 2}F_{n,k}t^{n} \\
&=&\left( kt+t^{2}\right) E(t) +(k-1)t+\sum_{n\geq
0}F_{n+1,k}t^{n}-\sum_{n\geq 0}F_{n,k}t^{n} \\
&=&\left( kt+t^{2}\right) E(t) +\left( \frac{1}{t}-1\right)
F(t) -1.
\end{eqnarray*}%
Thus from this identity and the generating function in~\eqref{eq:gen-fn-k-fib} the proposition follows.
\end{proof}

\begin{proposition}
The size of $\Pi _{n,k}$ is%
\begin{equation*}
|E(\Pi _{n,k})|=\sum_{i=0}^{n}F_{i,k}\left( F_{n-i+2,k}-F_{n-i+1,k}\right) .
\end{equation*}
\end{proposition}

\begin{proof}
From Proposition~\ref{prop:gf-edges}, we have%
\begin{equation*}
E(t)=t^{-1}\left( k-1+t\right) F^{2}(t).
\end{equation*}%
From the product of formal power series, we have
\begin{equation*}
F^{2}(t) =\sum_{n=0}^{\infty
}\sum_{i=0}^{n}F_{i,k}F_{n-i,k}t^{n}.
\end{equation*}%
Hence we can compute as follows:
\begin{align*}
E(t) &=\left( \left( k-1\right) t^{-1}+1\right) \left( \sum_{n=0}^{\infty
}\sum_{i=0}^{n}F_{i,k}F_{n-i,k}t^{n}\right) \\
&=\left( k-1\right) \sum_{n=0}^{\infty
}\sum_{i=0}^{n+1}F_{i,k}F_{n-i+1,k}t^{n}+\sum_{n=0}^{\infty
}\sum_{i=0}^{n}F_{i,k}F_{n-i,k}t^{n} \\
&=\sum_{n=0}^{\infty }\sum_{i=0}^{n}F_{i,k}\left( \left( k-1\right)
F_{n-i+1,k}+F_{n-i,k}\right) t^{n} \\
&=\sum_{n=0}^{\infty }\sum_{i=0}^{n}F_{i,k}\left(
kF_{n-i+1,k}+F_{n-i,k}-F_{n-i+1,k}\right) t^{n} \\
&=\sum_{n=0}^{\infty }\sum_{i=0}^{n}F_{i,k}\left(
F_{n-i+2,k}-F_{n-i+1,k}\right) t^{n}.
\end{align*}
\end{proof}

Using the fundamental decomposition~\eqref{decom} and the same methods used in the case of Fibonacci cubes (see~\cite{FibonacciBook-2023}), the following holds. We omit the proof.

\begin{proposition}
    \label{prop:hamiltonian-path}
    For every $n \geq 0$ and $k \geq 2$, the graph $\Pi_{n,k}$ has a Hamiltonian path.
\end{proposition}

Since $\Pi_{n,k}$ is bipartite and when $n$ is even the partite sets are of different size, $\Pi_{n,k}$ has no Hamiltonian cycle if $n$ is even.  If $n$ is odd, it is not obvious which graphs $\Pi_{n,k}$ are Hamiltonian and which are not (see Problem~\ref{prob:hamilton}).

\section{Radius and diameter}
\label{sec:rad-diam}

If $t$ is a word over alphabet $\Sigma$, then $|t|_i$ denotes the number of occurrences of the letter $i \in \Sigma$ in the word $t$. A substring consisting of $m$ consecutive letters $i \in \Sigma$ is denoted by $i^m$.

\begin{proposition}
	\label{prop:radius}
	If $n \geq 1$ and $k \geq 2$, then $$\rad(\Pi_{n,k}) = \left \lfloor \frac{k n}{2} \right \rfloor.$$
\end{proposition}

\begin{proof} 
Let $t = t_1 \ldots t_n \in V(\Pi_{n,k})$. We distinguish two cases.
	
\begin{description}
	\item[Case 1] $k$ is even.\\
		Let $|t|_k = 2\ell$. Then $|t|_0 + \cdots + |t|_{k-1} = n - 2 \ell$. Consider the generalized Pell string $t'$ obtained from $t$ by first exchanging the role of $i$ and $i + \frac{k}{2}$ for every $i \in \{0, \ldots, \frac{k}{2}-1\}$, and then replacing each substring $kk$ with $00$. Exchanging the role of $i$ and $i + \frac{k}{2}$ requires $\frac{k}{2}$ consecutive changes in the string, while replacing $kk$ with $00$ requires $2k-1$ consecutive changes (for example, $kk \to (k-1)(k-1) \to (k-2)(k-1) \to \cdots \to 00$). Since these changes are disjoint, we obtain $$d(t,t') = (n-2 \ell) \frac{k}{2} + \ell (2 k - 1) = \frac{kn}{2} + \ell (k-1) \geq \frac{k n}{2},$$
		since $\ell \geq 0$ and $k \geq 2$.
		
	\item[Case 2] $k$ is odd.\\
		Let $|t|_k = 2\ell$, $|t|_{\frac{k-1}{2}} = m$, and let $p \geq 0$ denote the maximum number of disjoint appearances of the substring $(\frac{k-1}{2})(\frac{k-1}{2})$ in $t$. Then $|t|_0 + \cdots + |t|_{k-1} - |t|_{\frac{k-1}{2}} = n - 2 \ell - m$ and since $p$ is the largest possible, $m \leq 2 p + \lceil \frac{n-2p}{2} \rceil$.
		
		Consider the generalized Pell string $t'$ obtained from $t$ by consecutively applying the following changes to the string $t$:
		\begin{enumerate}[(i)]
			\item exchange the role of $i$ and $i + \frac{k+1}{2}$ for every $i \in \{0, \ldots, \frac{k-3}{2}\}$; \label{item1}
			\item replace each substring $kk$ with $00$; \label{item2}
			\item replace each of the $p$ disjoint pairs of $(\frac{k-1}{2}) (\frac{k-1}{2})$ with $kk$; and \label{item3}
			\item replace each remaining $\frac{k-1}{2}$ with $0$. \label{item4}
		\end{enumerate}
		
		Each exchange from (\ref{item1}) requires $\frac{k+1}{2}$ consecutive changes in the string, each replacement from (\ref{item2}) requires $2k-1$ consecutive changes, each replacement from (\ref{item3}) needs $k$ changes, and each replacement from (\ref{item4}) needs $\frac{k-1}{2}$ changes. Since these changes are disjoint, we obtain 
  \begin{align*}
  d(t,t') & = (n-2 \ell - m) \frac{k+1}{2} + \ell (2 k - 1) + p k + (m-2p) \frac{k-1}{2} \\
  & = \frac{kn+n}{2} + \ell (k-2) -m + p.    
  \end{align*}
		Using the fact that $\ell \geq 0$, $k \geq 3$, and $m \leq 2 p + \lceil \frac{n-2p}{2} \rceil$, we get $$d(t, t') \geq \frac{kn+n}{2} - \left \lceil \frac{n-2p}{2} \right \rceil - p.$$
		
		If $n$ is even, then this yields $d(t, t') \geq \frac{kn}{2} = \left \lfloor \frac{k n}{2} \right \rfloor$, while if $n$ is odd, we get $d(t, t') \geq \frac{kn-1}{2} = \left \lfloor \frac{k n}{2} \right \rfloor$.
\end{description}
Thus $\rad(\Pi_{n,k}) \geq \left \lfloor \frac{k n}{2} \right \rfloor$. To prove the equality it suffices to find a vertex with eccentricity $\left \lfloor \frac{k n}{2} \right \rfloor$.  We claim that if $k$ is even, then $t = \left(\frac{k}{2} \right)^n$ is such a vertex, and if $k$ is odd, then $t' = \left(\frac{k-1}{2} \right)^n$ is a required vertex. Indeed, if $k$ is even, then $d(t, 0^n) = \left \lfloor \frac{k n}{2} \right \rfloor$ and by the above $d(t, x) \leq \left \lfloor \frac{k n}{2} \right \rfloor$ for any other vertex $x \in V(\Pi_{n,k})$, hence $\ecc(t) = \left \lfloor \frac{k n}{2} \right \rfloor$. Similarly, if $k$ is odd and $n$ is even, then $d(t', k^n) = \left \lfloor \frac{k n}{2} \right \rfloor$, and if $k$ is odd and $n$ is odd, then $d(t', k^{n-1}0) = \left \lfloor \frac{k n}{2} \right \rfloor$. Hence, $\ecc(t') = \left \lfloor \frac{k n}{2} \right \rfloor$ as claimed.
\end{proof}

The center of the Pell graph $\Pi_n$ is isomorphic to the Fibonacci cube $\Gamma_n$ \cite[Proposition 5]{munarini-2019}. It turns out that the same happens for certain generalized Pell graphs (see Theorem~\ref{thm:center}), but not for every $k \geq 2$. Using a computer, we have computed the cardinalities of the center of some small generalized Pell graphs. These are given in Table~\ref{table:center-size}.

\begin{table}[h!]
	\centering
	\begin{tabular}{c | c  c c c c c c c c c } 
		\diagbox{$k$}{$n$} & 1 & 2 & 3 & 4 & 5 & 6 & 7 & 8 & 9 & 10 \\
		\hline
		2 & 2 & 3 & 5 & 8 & 13 & 21 & 34 & 55 & 89 & 144 \\ 
		3 & 1 & 3 & 2 & 8 & 4 & 20 & 8 & 48 & 16 & 112 \\
        4 & 2 & 3 & 5 & 8 & 13 &  & &  &  &  \\ 
		5 & 1 & 3 & 2 & 8 & 4 &  & &  &  & \\
        6 & 2 & 3 & 5 & 8 & 13 &  & &  &  & \\ 
		7 & 1 & 3 & 2 & 8 & 4 &  & &  &  & \\
        8 & 2 & 3 & 5 & 8 & 13 &  & &  &  & \\ 
		9 & 1 & 3 & 2 & 8 & 4 &  & &  &  & \\
	\end{tabular}
	\caption{The cardinality of the center of some of the graphs $\Pi_{n,k}$.}
	\label{table:center-size}
\end{table}

The values in Table~\ref{table:center-size} indicate that $|C(\Pi_{n,k})|$ depends only on the parity of $k$, and not its exact value. In the following we prove that this is indeed the case, and explicitly describe the center of generalized Pell graphs. For this, we introduce the following families of words.

Let $k \geq 2$ be even and set $$\Theta_n(k) = \left \{ t = t_1 \ldots t_n; \; t_i \in \left\{ \frac{k}{2}-1, \frac{k}{2} \right\} \text{ and $t$ contains no two consecutive $\frac{k}{2}-1$}  \right \}.$$ It is easy to see that $\Pi_{n,k}[\Theta_n(k)] \cong \Gamma_n$.

For $n$ even define $\Phi_{n}(a,b)$ to be the set of words of length $n$ over the alphabet $\{ aa, ab, ba \}$, where the letter $ba$ never appears before the letter $ab$. For example, $\underline{ba}aaaa\underline{ab}aa \notin \Phi_{10}(a,b)$, but $ababbaaaba \in \Phi_{10}(a,b)$. Note that a word of length $n$ consists of $n/2$ letters.

For $n$ odd define $\Psi_{n}(a,b)$ to be the set of words of length $n$ over the alphabet $\{a,b\}$ that start and end with $a$, contain no substring $bb$, and have all runs of $a$s of odd length. For example, $abb \notin \Psi_{3}(a,b)$, $baaba \notin \Psi_{5}(a,b)$ and $abaaa \in \Psi_{5}(a,b)$. 

\begin{theorem}
    \label{thm:center}
    If $k \geq 2$ and $n \geq 2$, then 
    $$C(\Pi_{n,k}) = \begin{cases}
        \Theta_n(k); & k \text{ even},\\
        \Phi_n(\frac{k-1}{2}, \frac{k+1}{2}); & k \text{ odd and } n \text{ even},\\
        \Psi_n(\frac{k-1}{2}, \frac{k+1}{2}); & k \text{ odd and } n \text{ odd}.
    \end{cases}$$
    Consequently,
    $$|C(\Pi_{n,k})| = \begin{cases}
        F_{n+2}; & k \text{ even},\\
        (n+4) 2^{\frac{n}{2}-2}; & k \text{ odd and } n \text{ even},\\
        2^{\frac{n-1}{2}}; & k \text{ odd and } n \text{ odd}.
    \end{cases}$$
    In addition, if $k$ is even, then $\Pi_{n,k}[C(\Pi_{n,k})] \cong \Gamma_n$.
\end{theorem}

\begin{proof} We distinguish between three main cases.

    \medskip\noindent \underline{$k\geq 2$ even}:\\
    We are going to prove that $C(\Pi_{n,k}) = \Theta_n(k)$. From this it immediately follows that $\Pi_{n,k}[C(\Pi_{n,k})] \cong \Gamma_n$ and that $|C(\Pi_{n,k})| = F_{n+2}$.

    Let $t = t_1 \ldots t_n \in V(\Pi_{n,k})$. If $t$ contains the substring $kk$, then reevaluating the calculation in Case 1 of the proof of Proposition \ref{prop:radius} for $\ell \geq 1$ yields $d(t, t') \geq \frac{kn}{2} + 1 > \rad(\Pi_{n,k})$, thus such $t$ is not in the center of $\Pi_{n,k}$. From now on we may thus assume that $t$ contains no $kk$.

    If $t$ contains $x \in \{0, \ldots, k-1\} \setminus\{\frac{k}{2}-1, \frac{k}{2}\}$, then consider $t'' \in V(\Pi_{n,k})$ obtained in the following way. First replace $x$ with $k-1$ if $x \leq \frac{k}{2} - 2$, or with $0$ if $x \geq \frac{k}{2} + 1$. Next, for each other letter in $t$, exchange $i$ and $i + \frac{k}{2}$, $i \in \{0, \ldots, \frac{k}{2}-1\}$. Then $d(t, t'') = \left( \frac{k}{2} + 1 \right) + (n-1) \frac{k}{2} = \frac{nk}{2} + 1 > \rad(\Pi_{n,k})$.
    
    If $t$ contains only letters $\frac{k}{2}-1$ and $\frac{k}{2}$, but it also contains at least one substring $(\frac{k}{2}-1)(\frac{k}{2}-1)$, then consider $t''' \in V(\Pi_{n,k})$ obtained in the following way. Replace this substring $(\frac{k}{2}-1)(\frac{k}{2}-1)$ with $kk$, and for the other letters in $t$, exchange $\frac{k}{2}-1$ with $k-1$ and $\frac{k}{2}$ with $0$. Clearly, $d(t, t''') = (1 + k) + (n-2) \frac{k}{2} = \frac{nk}{2} + 1 > \rad(\Pi_{n,k}).$
    
    These arguments show that $C(\Pi_{n,k}) \subseteq \Theta_n(k)$. For $t \in \Theta_n(k)$, we prove that $\ecc(t) \leq \frac{nk}{2}$. Let $u = u_1 \ldots u_n \in V(\Pi_{n,k})$. If $u_i \in \{0, 1, \ldots, k-1\}$, then changing $t_i$ to $u_i$ requires at most $\frac{k}{2}$ steps. If $u_i u_{i+1} = kk$, then since $t$ contains no two consecutive $\left( \frac{k}{2} -1\right)$s, the change from $t_i t_{i+1}$ to $kk$ requires at most $k = 2 \cdot \frac{k}{2}$ steps. Thus $d(t,u) \leq n \cdot \frac{k}{2} = \rad(\Pi_{n,k})$. Thus $C(\Pi_{n,k}) = \Theta_n(k)$.
    
    \medskip\noindent \underline{$k \geq 3$ odd and $n \geq 2$ even}:\\
    We first prove that $C(\Pi_{n,k}) = \Phi_n(\frac{k-1}{2}, \frac{k+1}{2})$. Let $a = \frac{k-1}{2}$ and $b = \frac{k+1}{2}$.

    Let $t = t_1 \ldots t_n \in V(\Pi_{n,k})$, $|t|_k = 2\ell$, $|t|_{a} = m$, and let $p \geq 0$ denote the number of appearances of the substring $aa$ in $t$ which originate from the letter $aa$.

    First we prove that if $t \in \Phi_n(a,b)$, then $t \in C(\Pi_{n,k})$. Since $t \in \Phi_n(a,b)$, $p$ equals the number of times the letter $aa$ is used in $t$. Since in the words $ab$ and $ba$ both $a$ and $b$ appear an equal number of times, we know that each of $a$ and $b$ appears in pairs $ab$ and $ba$ exactly $\frac12 (n-2p)$ times, and thus $m = 2p + \frac12 (n-2p)$.
    
    If $u \in V(\Pi_{n,k})$, then 
    $$d(t,u)  \leq p k + \frac12 (n-2p) \frac{k-1}{2} + \frac12 (n-2p) \frac{k+1}{2} = \frac{nk}{2} = \rad(\Pi_{n,k}),$$
    since replacing $aa$ with $kk$ requires $k$ steps, replacing $a$ with $i$, $i \neq k$, requires at most $\frac{k-1}{2}$ steps, replacing $bb$ with $kk$ requires $k-2 \leq 2 \frac{k+1}{2}$ steps, replacing $b$ with $i$, $i \neq k$, requires at most $\frac{k+1}{2}$ steps, and replacing $ab$ or $ba$ with $kk$ requires at most $k-1 < \frac{k-1}{2} + \frac{k+1}{2}$ steps. This shows that $t \in C(\Pi_{n,k})$.
    
    Next we prove that if $t \notin \Phi_n(a,b)$, then $t \notin C(\Pi_{n,k})$. We consider the following cases.
    
    \begin{description}
    	\item[Case 1.] $t$ contains $kk$.\\
    		Let $t'$ be as in the proof of Proposition~\ref{prop:radius}. Then since $\ell \geq 1$ and $k \geq 3$, $d(t, t') \geq \frac{nk}{2} + 1 > \rad(\Pi_{n,k}).$
    		
    	\item[Case 2.] $t$ does not contain $kk$.\\
    	Since $n$ is even, letters in $t$ can be paired as $t_{2i-1} t_{2i}$ for $i \in [\frac{n}{2}]$. We will call this partition the \emph{pair-partition of $t$}.
    	\begin{description}
    		\item[Case 2.1.] $t$ contains $x$, $x \in \{0, \ldots, \frac{k-5}{2}, \frac{k+3}{2}, \ldots, k-1\}$.\\
    			Let $t'$ be as in the proof of Proposition~\ref{prop:radius}, except that $x$ is replaced by $k-1$ if $x \leq \frac{k-5}{2}$ and by $0$ if $x \geq \frac{k+3}{2}$. Then since $\ell = 0$, and replacing $x$ with $k-1$ or $0$ requires at least $\frac{k+3}{2}$ steps, we obtain 
                \begin{align*}
                    d(t,t') & \geq (n-m-1) \frac{k+1}{2} + \frac{k+3}{2} + pk + (m-2p) \frac{k-1}{2} \\
                    & = \frac{nk + n}{2} + 1 - m + p.
                \end{align*}
    			Recall from the proof of Proposition~\ref{prop:radius} that $m \leq 2 p + \frac{n-2p}{2}$, thus
    			$$d(t,t') \geq \frac{nk}{2} + 1 > \rad(\Pi_{n,k}).$$
    			
    		\item[Case 2.2.] $t$ contains only $\frac{k-3}{2}, a, b$, and $\frac{k-3}{2}$ appears at least once or $bb$ appears in the pair-partition of $t$.\\
    			Let $t''$ be constricted from the pair-partition of $t$ by the following:
    			\begin{enumerate}
    				\item replace each pair $aa$ with $kk$ (requires $k$ steps),
    				\item replace each pair consisting of $a$ and $\frac{k-3}{2}$ with $kk$ (requires $k+1 > k$ steps),
    				\item replace each pair $bb$ with $00$ (requires $k+1 > k$ steps),
    				\item replace each pair consisting of $a$ and $b$ with $k-1$ and $0$ (requires at least $k$ steps),
    				\item replace each pair consisting of $b$ and $\frac{k-3}{2}$ with $0$ and $k-1$ (requires at least $k+1 > k$ steps),
    				\item replace each pair $(\frac{k-3}{2}) (\frac{k-3}{2})$ with $kk$ (requires $k+2 > k$ steps).
    			\end{enumerate}
    			Since $\frac{k-3}{2}$ appears at least once or $bb$ appears as some $t_{2i-1} t_{2i}$, $i \in [\frac{n}{2}]$, in $t$, we get 
    			$$d(t, t'') \geq (k+1) + \left(\frac{n}{2} - 1 \right) k = \frac{nk}{2} + 1 > \rad(\Pi_{n,k}).$$
    			
    		\item[Case 2.3.] $t$ contains only $a$ and $b$, $bb$ does not appear in the pair-partition of $t$, and $ba$ appears at least once before $ab$ in the pair-partition of $t$.\\
    			Let $q$ be the number of times $aa$ appears in the pair-partition of $t$. Then since $ba$ appears before $ab$ at least once, $p \geq q+1$. The number of $b$s in $t$ is $\frac12 (n-2q)$ and $m = 2q + \frac12 (n-2q) = \frac{n}{2} + q$. Let $t'$ be obtained from $t$ as in the proof of Proposition~\ref{prop:radius}. Then we get
    			$$d(t, t') = \frac12 (n-2q) \frac{k+1}{2} + p k + (m-2p) \frac{k-1}{2} \geq \frac{nk}{2} + 1 > \rad(\Pi_{n,k}).$$
    			
    	\end{description}
    \end{description}
    Since we found a vertex of $\Pi_{n,k}$ that is at distance strictly more than $\rad(\Pi_{n,k})$ from $t$ in each case, it follows that $t \notin C(\Pi_{n,k})$. Thus $C(\Pi_{n,k}) = \Phi_n(a,b)$.

    It remains to prove that $|\Phi_n(a,b)| = (n+4) 2^{\frac{n}{2}-2}$. For $n \geq 4$ even, we have that $|\Phi_n(a,b)| = 2 |\Phi_{n-2}(a,b)| + 2^{\frac{n-2}{2}}$ since a word from $\Phi_n(a,b)$ either starts with $aa$ and is followed by a word from $\Phi_{n-2}(a,b)$, or starts with $ab$ and is followed by a word from $\Phi_{n-2}(a,b)$, or starts with $ba$ and is followed by a word over the alphabet $\{aa, ba\}$. Then the claim follows by induction.
    
    \medskip\noindent\underline{$k \geq 3$ odd and $n\geq 3$ odd}:\\
    We first prove that $C(\Pi_n,k) = \Psi_n(\frac{k-1}{2}, \frac{k+1}{2})$. Again, let $a = \frac{k-1}{2}$ and $b = \frac{k+1}{2}$. 
    
    Let $t = t_1 \ldots t_n \in V(\Pi_{n,k})$, $|t|_k = 2\ell$, $|t|_{a} = m$, and let $p \geq 0$ denote the maximum number of disjoint appearances of the substring $aa$ in $t$.
    
    Suppose that $t \in \Psi_n(a,b)$. Let $r$ denote the number of runs of $a$s in $t$. Then since each run of $a$s is of odd length we get $p = \frac{m-r}{2}$, and since there is no $bb$, $t_1 \neq b$, and $t_n \neq b$, there is exactly $r-1$ $b$s in $t$, so we have $m = n - r + 1$.
    
    If $u \in V(\Pi_{n,k})$, then 
    $$d(t,u)  \leq p k + (m-2p) \frac{k-1}{2} + (r-1) \frac{k+1}{2} = \frac{nk-1}{2} = \rad(\Pi_{n,k}),$$
    since replacing $aa$ with $kk$ requires $k$ steps, replacing $a$ with $i$, $i \neq k$, requires at most $\frac{k-1}{2}$ steps, replacing $b$ with $i$, $i \neq k$, requires at most $\frac{k+1}{2}$ steps, and replacing $ab$ or $ba$ with $kk$ requires at most $k-1 < \frac{k-1}{2} + \frac{k+1}{2}$ steps. This shows that $t \in C(\Pi_{n,k})$.
    
    Now suppose that $t \notin \Phi_n(a,b)$. To see that $t \notin C(\Pi_{n,k})$, we consider the following cases.
    
    \begin{description}
    	\item[Case 1.] $t$ contains $kk$.\\
    	Let $t'$ be as in the proof of Proposition~\ref{prop:radius}. Then since $\ell \geq 1$ and $k \geq 3$, $d(t, t') \geq \frac{nk-1}{2} + 1 > \rad(\Pi_{n,k}).$
    	
    	\item[Case 2.] $t$ does not contain $kk$.\\
    	Letters in $t$ can be partitioned into pairs $t_{2i-1} t_{2i}$ for $i \in [\frac{n-1}{2}]$, and a singleton $t_n$. We will call this partition the \emph{pair-partition of $t$}.
    	\begin{description}
    		\item[Case 2.1.] $t$ contains $x$, $x \in \{0, \ldots, \frac{k-5}{2}, \frac{k+3}{2}, \ldots, k-1\}$.\\
    		Let $t'$ be as in the proof of Proposition~\ref{prop:radius}, except that $x$ is replaced by $k-1$ if $x \leq \frac{k-5}{2}$ and by $0$ if $x \geq \frac{k+3}{2}$. Then since $\ell = 0$, and replacing $x$ with $k-1$ or $0$ requires at least $\frac{k+3}{2}$ steps, we obtain 
            \begin{align*}
                d(t,t') & \geq (n-m-1) \frac{k+1}{2} + \frac{k+3}{2} + pk + (m-2p) \frac{k-1}{2} \\
                & = \frac{nk + n}{2} + 1 - m + p.
            \end{align*}
    		Recall from the proof of Proposition~\ref{prop:radius} that $m \leq 2 p + \frac{n-2p+1}{2}$, thus
    		$$d(t,t') \geq \frac{nk-1}{2} + 1 > \rad(\Pi_{n,k}).$$
    		
    		\item[Case 2.2.] $t$ contains only $\frac{k-3}{2}, a, b$, but $\frac{k-3}{2}$ appears at least once or $bb$ appears at least once or $t_1 = b$ or $t_n = b$.\\
    		If $t_1 = b$ and $t_n \neq b$, then without loss of generality exchange the role of $t_i$ and $t_{n-i}$ for all $i \in [\frac{n-1}{2}]$. Let $t''$ be constricted from the pair-partition of $t$ by the following:
    		\begin{enumerate}
    			\item replace each pair $aa$ with $kk$ (requires $k$ steps),
    			\item replace each pair consisting of $a$ and $\frac{k-3}{2}$ with $kk$ (requires $k+1 > k$ steps),
    			\item replace each pair $bb$ with $00$ (requires $k+1 > k$ steps),
    			\item replace each pair consisting of $a$ and $b$ with $k-1$ and $0$ (requires at least $k$ steps),
    			\item replace each pair consisting of $b$ and $\frac{k-3}{2}$ with $0$ and $k-1$ (requires at least $k+1 > k$ steps),
    			\item replace each pair $(\frac{k-3}{2}) (\frac{k-3}{2})$ with $kk$ (requires $k+2 > k$ steps),
    			\item replace $t_n$ with $0$ if $t_n = b$ and with $k-1$ otherwise (requires at least $\frac{k-1}{2}$ steps, but $\frac{k+1}{2}$ if $t_n \in \{\frac{k-3}{2}, b\}$).
    		\end{enumerate}
    		Since $\frac{k-3}{2}$ appears at least once or $bb$ appears as some pair or $t_n = b$, we get 
    		$$d(t, t'') \geq \frac{n-1}{2} k + \frac{k-1}{2} + 1 = \frac{nk-1}{2} + 1 > \rad(\Pi_{n,k}).$$
    		
    		\item[Case 2.3.] $t$ contains only $a$ and $b$, $bb$ does not appear in $t$, $t_1 \neq b$, $t_n \neq b$, but not all runs of $a$s are of odd length.\\
    		Let $r$ be the number of runs of $a$s. Then since not all runs of $a$s are of odd length, $p \geq \frac{m-r+1}{2}$. The number of $b$s in $t$ is $r-1$ and $m = n - r + 1$. Let $t'$ be obtained from $t$ as in the proof of Proposition~\ref{prop:radius}. Then we get
    		$$d(t, t') = (r-1) \frac{k+1}{2} + p k + (m-2p) \frac{k-1}{2} \geq \frac{nk-1}{2} + \frac12 > \rad(\Pi_{n,k}).$$
    		
    	\end{description}
    \end{description}
    Since we found a vertex of $\Pi_{n,k}$ that is at distance strictly more than $\rad(\Pi_{n,k}) = \frac{nk-1}{2}$ from $t$ in each case, it follows that $t \notin C(\Pi_{n,k})$. Thus $C(\Pi_{n,k}) = \Psi_n(a,b)$.

    It remains to show that $|\Psi_n(a,b)| = 2^{\frac{n-1}{2}}$. For $n \geq 3$ odd, we have that $|\Psi_n(a,b)| = 2 |\Psi_{n-2}(a,b)|$ since the word can either start by $aa$ or $ab$, and in both cases it needs to be followed by a word from $\Psi_{n-2}(a,b)$ (in particular, it needs to start with $a$). Thus the claim can be proved using induction on $n$.    
\end{proof}

We have excluded the case $n = 1$ from Theorem~\ref{thm:center}. But since $\Pi_{1,k} $ is isomorphic to the path on $k$ vertices, its center is isomorphic to either $K_1$ or $K_2$. An example of a generalized Pell graph with its center is presented in Figure~\ref{fig:pi-4-3}.

\begin{figure}[hp]
    \centering
    \begin{tikzpicture}[scale=0.7]
   
    \begin{scope}
    [
    vert/.style={circle, draw, fill=black!10, inner sep=0pt, minimum width=4pt}, 
    central/.style={circle, draw, fill=black, inner sep=0pt, minimum width=4pt},
    ]
        \foreach \p in {0,1,2}{
        \foreach \z in {0,1,2}
            \foreach \y in {0,1,2}
                \foreach \x in {0,1,2}
                    \node[vert] (\p+\z+\y+\x) at (5*\z+\x+0.3*\y,\y+0.7*\x-6*\p) {};
        \foreach \z in {0,1,2}
            \node[vert] (\p+\z+3+3) at (5*\z+2+0.5+0.3,3+0.7*2-6*\p) {};
        \foreach \x in {0,1,2}
            \node[vert] (\p+3+3+\x) at (5*3+\x+0.3, 2+0.7*\x-6*\p) {};
            
        \foreach \z in {0,1,2}
            \foreach \y in {0,1,2}
                \foreach \x [remember=\x as \lastx (initially 0)] in {1,2}
                    \path (\p+\z+\y+\x) edge (\p+\z+\y+\lastx);

        \foreach \z in {0,1,2}
            \foreach \y [remember=\y as \lasty (initially 0)] in {0,1,2}
                \foreach \x in {0,1,2}
                    \path (\p+\z+\y+\x) edge (\p+\z+\lasty+\x);
                    
        \foreach \z [remember=\z as \lastz (initially 0)] in {1,2}
            \foreach \y in {0,1,2}
                \foreach \x in {0,1,2}
                    \path[gray] (\p+\z+\y+\x) edge (\p+\lastz+\y+\x);

        \foreach \z [remember=\z as \lastz (initially 0)] in {1,2}
            \path[gray] (\p+\z+3+3) edge (\p+\lastz+3+3);

        \foreach \z in {0,1,2}
            \path (\p+\z+2+2) edge (\p+\z+3+3);

        \foreach \x [remember=\x as \lastx (initially 0)] in {1,2}
            \path (\p+3+3+\x) edge (\p+3+3+\lastx);

        \foreach \x in {0,1,2}
            \path[gray] (\p+2+2+\x) edge (\p+3+3+\x);
        }

        \foreach \p [remember=\p as \lastp (initially 0)] in {1,2}{
            \foreach \x in {0,1,2}
                \foreach \y in {0,1,2}
                    \foreach \z in {0,1,2}
                        \path[gray] (\p+\x+\y+\z) edge (\lastp+\x+\y+\z);
        }

        \foreach \p [remember=\p as \lastp (initially 0)] in {1,2}{
            \foreach \x in {0,1,2}{
                \path[gray] (\p+\x+3+3) edge (\lastp+\x+3+3);
                }
        }

        \foreach \p [remember=\p as \lastp (initially 0)] in {1,2}{
            \foreach \x in {0,1,2}{
                \path[gray] (\p+3+3+\x) edge (\lastp+3+3+\x);
                }
        }

        \foreach \y in {0,1,2}
            \foreach \x in {0,1,2}
                \node[vert] (3+2+\y+\x) at (5*2+\x+0.3*\y,\y+0.7*\x-6*3) {};
        \foreach \z in {0,1,2}
            \node[vert] (3+2+3+3) at (5*2+2+0.5+0.3,3+0.7*2-6*3) {};

        \foreach \y in {0,1,2}
            \foreach \x [remember=\x as \lastx (initially 0)] in {1,2}
                \path (3+2+\y+\x) edge (3+2+\y+\lastx);

        \foreach \x in {0,1,2}
            \foreach \y [remember=\y as \lasty (initially 0)] in {1,2}
                \path (3+2+\y+\x) edge (3+2+\lasty+\x);

        \path (3+2+3+3) edge (3+2+2+2);
        \path[gray] (3+2+3+3) edge (2+2+3+3);

        \foreach \x in {0,1,2}
            \foreach \y in {0,1,2}
                \path[gray] (3+2+\y+\x) edge (2+2+\y+\x);

        \node (0mark) at (-3,1) {$0\_\_\_$};
        \node (1mark) at (-3,-5) {$1\_\_\_$};
        \node (2mark) at (-3,-11) {$2\_\_\_$};
        \node (33mark) at (-3,-16) {$33\_\_$};

        \node (0mark') at (1.4,6) {$\_0\_\_$};
        \node (1mark') at (6.4,6) {$\_1\_\_$};
        \node (2mark') at (11.4,6) {$\_2\_\_$};
        \node (33mark') at (16.2,6) {$\_33\_$};
        
        \node[central] (c1221) at (11+0.6,-4+0.7) {};
        \node[central] (c1211) at (11+0.3,-6+0.7*1+1) {};
        \node[central] (c1212) at (12+0.3,-6+0.7*2+1) {};
        \node[central] (c1121) at (6+0.6,-6+0.7*1+2) {};
        \node[central] (c1111) at (6+0.3,-6+0.7*1+1) {};
        \node[central] (c1112) at (7+0.3,-6+0.7*2+1) {};
        \node[central] (c2121) at (6+0.6,-12+0.7*1+2) {};
        \node[central] (c2111) at (6+0.3,-12+0.7*1+1) {};

        \path[black, line width = 0.5mm] (c1221) edge (c1121);
        \path[black, line width = 0.5mm] (c1221) edge (c1211);
        \path[black, line width = 0.5mm] (c1111) edge (c1121);
        \path[black, line width = 0.5mm] (c2121) edge (c1121);
        \path[black, line width = 0.5mm] (c2121) edge (c2111);
        \path[black, line width = 0.5mm] (c1211) edge (c1111);
        \path[black, line width = 0.5mm] (c1211) edge (c1212);
        \path[black, line width = 0.5mm] (c1111) edge (c1112);
        \path[black, line width = 0.5mm] (c1111) edge (c2111);
        \path[black, line width = 0.5mm] (c1212) edge (c1112);

        \node[vert, label=left:\footnotesize{$0000$}] (0000) at (0,0) {};
        \node[vert,, label=below right:\footnotesize{$0001$}] (0001) at (1,0.7) {};
        \node[vert,, label=below right:\footnotesize{$0002$}] (0002) at (2,1.4) {};
        \node[vert, label=left:\footnotesize{$0010$}] (0010) at (0.3,1) {};
        \node[vert,, label=right:\tiny{$0011$}] (0011) at (1.3,1.7) {};
        \node[vert,, label=below right:\footnotesize{$0012$}] (0012) at (2.3,2.4) {};
        \node[vert, label=left:\footnotesize{$0020$}] (0020) at (0.6,2) {};
        \node[vert,, label=above left:\footnotesize{$0021$}] (0021) at (1.6,2.7) {};
        \node[vert, label=below right:\footnotesize{$0022$}] (0022) at (2.6,3.4) {};
        \node[vert, label=left:\footnotesize{$0033$}] (0033) at (2.8,4.4) {};
        \end{scope}
    \end{tikzpicture}
    \caption{The graph $\Pi_{4,3}$ with its center marked in black. Notice that in this case, the center is isomorphic to a Fibonacci cube (however, this is not the case in general for $k$ odd).}
    \label{fig:pi-4-3}
\end{figure}

\begin{proposition}
    \label{prop:diameter}
    If $n \geq 1$ and $k \geq 2$, then $$\diam(\Pi_{n,k}) = n k - \left\lceil \frac{n}{2} \right\rceil.$$
\end{proposition}

\begin{proof}
    If $n$ is even, then $d(0^n, k^n) = n(k-1) + \frac{n}{2} = nk - \frac{n}{2}$. If $n$ is odd, then $d(0^n, k^{n-1}(k-1)) = n(k-1) + \frac{n-1}{2} = nk - \frac{n+1}{2}$. Thus $\diam(\Pi_{n,k}) \geq nk - \lceil \frac{n}{2} \rceil$. We need to prove that this is also the upper bound.

    For $t = t_1\ldots t_n \in V(\Pi_{n,k})$, $\ecc(t) \leq n (k-1) + \lfloor \frac{n}{2} \rfloor$, since each $t_i$ can contribute at most $k-1$ to the distance by itself, and at most one more in a pair with $t_{i-1}$ or $t_{i+1}$ (but these pairs need to be disjoint).
\end{proof}

Additionally, it is not difficult to see that if $n$ is even, the periphery of $\Pi_{n,k}$ consists only of vertices obtained by using strings $00$ and $kk$. If $n$ is odd, the periphery is formed by vertices consisting of strings $00$, $kk$, and exactly one additional occurrence of either $0$ or $k-1$.

\section{Additional properties}
\label{sec:additional}

\subsection{The cube polynomial}

The cube polynomial of a graph $G$ is denoted by $C_G (x)$, and
is the generating function $C_ G(x) =\sum_{i\geq 0}c_{i}(G) x^{i},$ where $c_{i}(G)$ counts the number of induced $%
i$-cubes in $G$. This polynomial was introduced in~\cite{bresar-2003}, see also~\cite{belbachir-2020, tratnik-2023}. Clearly, $c_{0}(G) =|V(G)|$ and $c_{1}(G) =|E(G)|.$

The first few cube polynomials of $\Pi _{n,k}$ are listed below: 
\begin{align*}
C_{\Pi_{0,k}}( x)  & = 1 \\
C_{\Pi_{1,k}}( x)  & = k+\left( k-1\right) x \\
C_{\Pi_{2,k}} ( x)  & = \left( \allowbreak k^{2}+1\right) +\left(
2k^{2}-2k+1\right) x+\left( k^{2}-2k+1\right) x^{2} \\
C_{\Pi _{3,k}}(x)  & = k^{3}+2k+\left( 3k^{3}-\allowbreak
3k^{2}+4k-2\right) x\allowbreak +\left( 3k^{3}-6k^{2}+5k-2\right)
x^{2} \\ 
&\quad + \left( k^{3}-3k^{2}+3k-1\right) x^{3} \\
C_{\Pi _{4,k}}(x)  & = \left( k^{4}+3k^{2}+1\right) + 
\left(4k^{4}-4k^{3}+9k^{2}-\allowbreak 6k+2\right) x \\
&\quad + \left( 6k^{4}+\allowbreak 15k^{2}-12k^{3}-12k+4\right) x^{2}  \\
& \quad + \left( 4k^{4}-\allowbreak 12k^{3}+15k^{2}-10k+3\right) x^{3} \\
& \quad + \left(k^{4}-4k^{3}+6k^{2}-4k+1\right) x^{4}
\end{align*}
The next result follows from the recursive structure of $\Pi _{n,k}$.

\begin{proposition}
The cube polynomials $C_{\Pi _{n,k}}(x) $ satisfy the 
recurrence relation
$$
C_{\Pi _{n,k}}(x) =\left( k+\left( k-1\right) x\right) C_{\Pi
_{n-1,k}}(x) +\left( 1+x\right) C_{\Pi _{n-2,k}}(x),\;n\geq 2,
$$
with the initial values $C_{\Pi _{0,k}}(x) =1$ and $C_{\Pi
_{1,k}}(x) =k+\left( k-1\right) x.$
\end{proposition}

From the recurrence relation of the cube polynomials, we can derive the following result using standard methods.

\begin{proposition}
The generating function of the sequence $\left \{ C_{\Pi _{n,k}}(x)
\right \} _{n\geq 0}$ is%
\begin{equation*}
\sum_{n\geq 0}C_{\Pi _{n,k}}(x) t^{n}=\frac{1}{1-\left( k+\left(
k-1\right) x\right) t-\left( 1+x\right) t^{2}}.
\end{equation*}
\end{proposition}


From the generating function of the cube polynomials, we get the following
result.

\begin{proposition}
For $n\geq 0,$ the cube polynomial $C_{\Pi_{n,k}}(x)$ is of degree $n$ and%
\begin{equation*}
C_{\Pi _{n,k}}(x) =\sum_{i=0}^{\left \lfloor \frac{n}{2}%
\right \rfloor }\binom{n-i}{i}\left( k+\left( k-1\right) x\right)
^{n-2i}\left( 1+x\right) ^{i}.
\end{equation*}
\end{proposition}

\subsection{Distribution of the degrees}
\label{sec:degrees}

The distribution of degrees of Pell graphs has been studied in~\cite{munarini-2019}. If $t \in V(\Pi_n)$, then $$\deg(t) = |t|_0 + |t|_1 + \frac12 |t|_2 + e,$$
where $e$ is the number of pairwise disjoint occurences of $11$ in $t$. In particular, for $n \geq 1$, $\Delta(\Pi_n) = 2n-1$ and $\delta(\Pi_n) = \left \lceil \frac{n}{2} \right \rceil$, see~\cite[Proposition 27]{munarini-2019}. Since generalized Pell graphs with $k=2$ are isomorphic to Pell graphs, we will only consider graphs $\Pi_{n,k}$ for $k \geq 3$.

\begin{proposition}
    \label{prop:degrees}
    If $n \geq 1$, $k \geq 3$ and $t \in V(\Pi_{n,k})$, then $$\deg(t) = |t|_0 +  2 \cdot \sum_{i=1}^{k-1} |t|_i + \frac12 |t|_k - r,$$ where $r$ is the number of runs of $(k-1)$s in $t$.
\end{proposition}

\begin{proof}
    Let $t = t_1 \ldots t_n \in V(\Pi_{n,k})$. If $t_i = 0$, its contribution to the degree of $t$ is 1, while if $t_i \in [k-2]$, it contributes 2 (since $k \geq 3$). It is also easy to see that each occurrence of $kk$ contributes 1. Note that if $t_i = t_{i+1} = t_{i+2} = t_{i+3} = k$ is the start of a run of $k$s in $t$, then only the pairs $t_i t_{i+1}$ and $t_{i+2} t_{i+3}$ will result in a valid generalized Pell string after $kk$ is exchanged with $(k-1)(k-1)$.
    
    Lastly, let $r$ be the number of runs of $(k-1)$s in $t$ and let $q_1, \ldots, q_r$ be lengths of these runs. Observe that if $t_i = k-1$, it contributes 1 to the degree of $t$ (by exchanging $k-1$ with $k-2$). However, each occurrence of $(k-1)(k-1)$ contributes an additional 1 to the degree of $t$ (by exchanging it with $kk$). Note that any pair of consecutive $(k-1)$s yields a generalized Pell string, thus a run of $(k-1)$s of length $q_j$ contributes $q_j - 1$. Then occurrences of $k-1$ in $t$ altogether contribute $$\sum_{i=1}^r q_i + \sum_{i=1}^r (q_i - 1) = 2 \sum_{i=1}^r q_i - r = 2 |t|_{k-1} - r.$$
    Combining the observed contributions yields the desired formula.
\end{proof}

\begin{corollary}
    \label{cor:deltas}
    If $n \geq 1$ and $k \geq 3$, then $\Delta(\Pi_{n,k}) = 2n$ and $\delta(\Pi_{n,k}) = \left \lceil \frac{n}{2} \right \rceil$.
\end{corollary}

\begin{proof}
    Let $t \in V(\Pi_{n,k})$. Since $\sum_{i=0}^k |t|_i = n$, it clearly holds 
    $$\left \lceil \frac{n}{2} \right \rceil \leq \deg(t) \leq 2n.$$ The left bound is attained for example by the vertex $k^n$ if $n$ is even, and by the vertex $k^{n-1}(k-1)$ if $n$ is odd. The right bound is attained for example by the vertex $1^n$ (since $k \geq 3$).
\end{proof}

\begin{corollary}
    \label{cor:Delta-1}
    The number of vertices of $\Pi_{n,k}$ of degree $\Delta(\Pi_{n,k}) - 1$ is $$n (n-1)^{k-2} + \sum_{\ell=1}^n (n-\ell+1) (n-\ell)^{k-2}.$$
\end{corollary}

\begin{proof}
    If $t \in V(\Pi_{n,k})$, then $\deg(t) = 2n-1$ if either $|t|_0 = 1$ and $|t|_{k-1} = |t|_k = 0$, or $|t|_0 = |t|_k = 0$, $|t|_{k-1} \geq 1$ and $t$ contains only one run of $(k-1)$s. A simple counting argument shows the formula.
\end{proof}

\subsection{Median graphs}

A {\em median} of a triple $x$, $y$, $z$ of vertices of a graph $G$ is a vertex $u$ that simultaneously lies on a shortest $x,y$-path, a shortest $x,z$-path, and a shortest $y,z$-path. $G$ is a {\em median graph} if each triple of vertices has a unique median~\cite{mulder-1978}.

From~\cite{klavzar-2005} we know that Fibonacci cubes are median graphs and from~\cite{munarini-2019} that Pell graphs also belong to this family of graphs. This property extends to all generalized Pell graphs as the next result asserts. 

\begin{proposition}
 \label{prp:median} 
 If $n\ge 1$ and $k\ge 2$, then $\Pi_{n,k}$ is a median graph. 
\end{proposition}

\begin{proof}
(Sketch) We proceed by induction, the result being clear for $n=1$ and all $k$ since  $\Pi_{1,k}$ is isomorphic to the path with $k$ vertices. Suppose $n\ge 2$ and consider $\Pi_{n,k}$. The part of its fundamental decomposition~\eqref{decom} $$X = 0\Pi _{n-1,k} \oplus 1\Pi _{n-1,k} \oplus \cdots \oplus \left( k-1\right) \Pi
_{n-1,k}$$
is isomorphic to the Cartesian product of $\Pi _{n-1,k}$ by the path on $k$ vertices. As the factors are median graphs by the induction hypothesis, and the Cartesian product operation preserves the property of being median; $X$ is also median. Finally, we can observe that $\Pi_{n,k}$ is obtained from $X$ by the so-called convex peripheral expansion (cf.~\cite{mulder-1978}), hence $\Pi_{n,k}$ is median as well.  
\end{proof}

\subsection{Subgraph of a Fibonacci cube}
\label{sec:subgraph-fib}

It is known \cite[Theorem 7]{munarini-2019} that the Pell graph $\Pi_n$ is a subgraph of the Fibonacci cube $\Gamma_{2n-1}$, written as $\Pi_n \subseteq \Gamma_{2n-1}$. We prove an analogous result for generalized Pell graphs. Notice that for $k=2$, Proposition~\ref{prop:subgraph-fib} states the same as the existing result for Pell graphs.

\begin{proposition}
    \label{prop:subgraph-fib}
    If $n \geq 1$ and $k \geq 2$, then $$\Pi_{n,k} \subseteq \Gamma_{(2k-2)n - 1}.$$
\end{proposition}

\begin{proof}
    We define a mapping $\varphi \colon \Pi_{n,k} \to \Gamma_{(2k-2)n-1}$ in the following way. First, consider a mapping $\varphi' \colon \Pi_{n,k} \to \Gamma_{(2k-2)n}$ that maps a vertex $t = t_1 \ldots t_n \in V(\Pi_{n,k})$ in the following way:
    \begin{align*}
        i & \mapsto (10)^{k-1-i} 0^{2i}\\
        kk & \mapsto 010^{2k-4}
    \end{align*}
    Clearly $\varphi'(t)$ is of length $(2k-2)n$, and contains no $11$, thus $\varphi'(t) \in V(\Gamma_{(2k-2)n})$. Observe also that $\varphi'(t)$ always ends with $0$. Deleting the ending $0$ gives $\varphi(t) \in V(\Gamma_{(2k-2)n - 1})$. By definition, $\varphi$ is injective. Thus it remains to show that it maps edges to edges.

    Let $pq \in E(\Pi_{n,k})$, where $p = p_1 \ldots p_n$ and $q = q_1 \ldots q_n$. If $p$ and $q$ differ in only one letter, then without loss of generality, $p_i = \ell$, $q_i = \ell+1$, and $p_j = q_j$ for all $j \in [n] \setminus \{j\}$, for some $1 \leq \ell \leq k-2$ and $1 \leq i \leq n$. Thus $\varphi(p)_{(2k-2)(i - 1) + 2 (k - \ell - 2)+1} = 1$, $\varphi(q)_{(2k-2)(i - 1) + 2 (k - \ell - 2)+1} = 0$ and $\varphi(p)_j = \varphi(q)_j$ for all $j \in [(2k-2)n - 1] \setminus \{(2k-2)(i - 1) + 2 (k - \ell - 2)+1\}$. So $\varphi(p) \varphi(q) \in E(\Gamma_{(2k-2)n-1})$.

    Otherwise, it must hold for some $i$, $1 \leq i \leq n-1$, that $p_i = p_{i+1} = k-1$, $q_i = q_{i+1} = k$, and $p_j = q_j$ for all $j \in [n] \setminus \{i, i+1\}$. Thus $\varphi(q)_{(2k-2) (i-1) + 2} = 1$, $\varphi(p)_{(2k-2)(i-1)+2} = 0$, and $\varphi(p)_j = \varphi(q)_j$ for all $j \in [(2k-2)n-1] \setminus \{(2k-2)(i-1)+2\}$. Hence again $\varphi(p) \varphi(q) \in E(\Gamma_{(2k-2)n-1})$.
\end{proof}

\section*{Concluding remarks}
\label{sec:conclude}

In the final stages of preparing our paper, we learned that Do\v{s}li\'{c} and Podrug had independently proposed a second generalization of Pell graphs, their approach was reported in~\cite{podrug-2023}. Their motivation was much as ours, that is, to construct graphs that reflect~\eqref{eq:F_n,k}, and doing so, they defined graphs denoted by $\Pi_n^k$. While these graph have the same order and size as the generalized Pell graphs $\Pi_{n,k}$ from this paper, their structure is significantly different. For instance, if $n$ and $k$ are both odd, then the center of $\Pi_n^k$ consists of a single vertex, while we have seen in Theorem~\ref{thm:center} that the center of $\Pi_{n,k}$ contains $2^{\frac{n-1}{2}}$ vertices. Additional structural differences include: 
\begin{itemize}
\item For $k = 2$, $\Pi_{n,2} \cong \Pi_n$, but $\Pi_n^2$ is not isomorphic to the Pell graph.
\item For $k \geq 2$ and $n \geq 3$, $\diam(\Pi_{n,k}) = nk - \lceil \frac{n}{2} \rceil < nk - 1 = \diam(\Pi_n^k)$.
\item For $n \geq 4$, graphs $\Pi_{n,k}$ contain strictly more vertices of degree $2n-1$ than graphs $\Pi_n^k$ (combining Corollary~\ref{cor:Delta-1} and an observation that $\Pi_n^k$ contains $2n(n-1)^{k-2} + (n-1) (n-2)^{k-2}$ vertices of degree $2n-1$, which are vertices $t$ with $|t|_0 = 1$, $|t|_{a-1} = |t|_a = 0$, or $|t|_{a-1} = 1$, $|t|_0 = |t|_a = 0$, or exactly one occurrence of $0(a-1)$ and otherwise containing only letters from $[k-2]$).
 \end{itemize}

We conclude the paper with some problems that appear interesting for further investigations.

As mentioned in Proposition~\ref{prop:hamiltonian-path}, it is easy to see that graphs $\Pi_{n,k}$ are traceable. The existence of a Hamiltonian cycle seems more complicated.

\begin{problem}
\label{prob:hamilton}
    Characterize graphs $\Pi_{n,k}$ that are hamiltonian.
\end{problem}

In Proposition~\ref{prop:subgraph-fib} we prove that $\Pi_{n,k}$ is a subgraph of a sufficiently large Fibonacci cube. However, it is not clear if the result is optimal.

\begin{problem}
    Does there exist a function $f(n,k) < (2k-2)n-1$ such that for every $n \geq 1$ and $k \geq 2$, $\Pi_{n,k} \subseteq \Gamma_{f(n,k)}$?
\end{problem}

It would be of interest to know the (edge) connectivity of generalized Pell graphs. For both of them we suspect that they are equal to the minimum degree (for every $n \geq 1$ and $k \geq 2$).

\begin{problem}
    Determine $\kappa(\Pi_{n,k})$ and $\kappa'(\Pi_{n,k})$.
\end{problem}

\section*{Acknowledgements}

This work has been supported by T\"{U}B\.{I}TAK and the Slovenian Research Agency under grant numbers 122N184 and BI-TR/22-24-20, respectively. 
Vesna Ir\v{s}i\v{c} and Sandi Klav\v{z}ar also acknowledge the financial support from the Slovenian Research Agency (research core funding P1-0297 and projects J1-2452 and N1-0285). 
Vesna Ir\v{s}i\v{c} also acknowledges the financial support from the ERC KARST project.

\end{document}